\documentclass[11pt]{amsart}

\usepackage{amsmath}				
\usepackage{amssymb}
\usepackage{amsthm}
\usepackage{amscd}
\usepackage{amsfonts}
\usepackage[all]{xy}

\usepackage{euler}
\usepackage[mathscr]{euscript}

\usepackage[
	pdftitle={Tropical compactification in log-regular varieties},
	pdfauthor={Martin Ulirsch},
	ocgcolorlinks,
	linkcolor=linkblue,
	citecolor=linkred,
	urlcolor=linkred]{hyperref}
\usepackage{color}
\definecolor{linkred}{rgb}{0.6,0,0}
\definecolor{linkblue}{rgb}{0,0,0.6}

\usepackage{tikz}
\usetikzlibrary{matrix}

\usepackage{enumerate}
\usepackage{color}
\usepackage{fullpage}

\newtheorem{theorem}{Theorem}[section]
\newtheorem{lemma}[theorem]{Lemma} 
\newtheorem{proposition}[theorem] {Proposition}
\newtheorem{corollary}[theorem]{Corollary}

\theoremstyle{definition}
\newtheorem{definition}[theorem]{Definition}
\newtheorem{example}[theorem]{Example}

\theoremstyle{remark}
\newtheorem{remark}[theorem]{Remark}

\newtheoremstyle{myitemstyle}						
	{}			
	{}			
	{}			
	{}			
	{}			
	{.}			
	{ }			
	{}			
\theoremstyle{myitemstyle}
\newtheorem{myitemthm}[theorem]{}

\newcommand{\R}{\mathbb{R}}
\newcommand{\Rbar}{\overline{\mathbb{R}}}
\newcommand{\Z}{\mathbb{Z}}
\newcommand{\Q}{\mathbb{Q}}
\newcommand{\C}{\mathbb{C}}
\newcommand{\N}{\mathbb{N}}

\newcommand{\calA}{\mathcal{A}}
\newcommand{\calO}{\mathcal{O}}

\DeclareMathOperator{\Spec}{Spec}
\DeclareMathOperator{\Hom}{Hom}
\DeclareMathOperator{\trop}{trop}
\DeclareMathOperator{\val}{val}

\DeclareMathOperator{\rk}{rk}
\DeclareMathOperator{\Trop}{Trop}

\DeclareMathOperator{\LOG}{LOG}
\DeclareMathOperator{\codim}{codim}

\DeclareMathOperator{\Link}{L}

\title{Tropical compactification in log-regular varieties}
\author{Martin Ulirsch}
\address{Department of Mathematics, Brown University, Providence, RI 02912, USA}
\email{\href{mailto:ulirsch@math.brown.edu}{ulirsch@math.brown.edu}}
\subjclass[2010]{14T05; 14G22; 20M14}
\date{\today}
\thanks{The author's research was supported in part by funds from BSF grant 201025 and NSF grants DMS0901278 and DMS1162367.} 

\begin{document}

\maketitle

\begin{abstract} 
In this article we define a natural tropicalization procedure for closed subsets of log-regular varieties in the case of constant coefficients and study its basic properties. This framework allows us to generalize some of Tevelev's results on tropical compactification as well as Hacking's result on the cohomology of the link of a tropical variety to log-regular varieties.
\end{abstract}

\setcounter{tocdepth}{1}

\tableofcontents


\section{Introduction}

Let $k$ be an algebraically closed field. In \cite{Kato_toricsing} K. Kato defines the notion of a \emph{log-regular} variety, a generalization of a toric variety in the category of logarithmic schemes over $k$. Such a log-regular variety $X$ contains an open and dense non-singular subset $X_0\subseteq X$ on whose complement $X-X_0$ at most toroidal singularities are allowed. Moreover $X-X_0$ has a natural stratification by locally closed subsets, called the \emph{strata} of $X$. These strata are in an order-reversing one-to-one correspondence with the positive dimensional cones in a rational polyhedral cone complex $\Sigma_X$ that is canonically associated to $X$. 

Let $Y$ be a closed subset of $X_0$. Based on the techniques developed in \cite{Ulirsch_functroplogsch} we define in this note a \emph{tropical variety} $\Trop_X(Y)$ associated to $Y$ relative to $X$ as a subset of $\Sigma_X$. 
The fundamental result on the structure of $\Trop_X(Y)$ (see \cite[Theorem A]{BieriGroves_polyhedral} and \cite[Theorem 2.2.3]{EinsiedlerKapranovLind_amoebae}) generalizes to this situation.

\begin{theorem}\label{thm_BieriGroves}
The tropical variety $\Trop_X(Y)$ admits a structure of a finite rational polyhedral cone complex of dimension $\leq\dim Y$. 
\end{theorem}

However, contrary to the classical result of Bieri and Groves \cite[Theorem A]{BieriGroves_polyhedral} as well as Einsiedler, Kapranov, and Lind \cite[Theorem 2.2.3]{EinsiedlerKapranovLind_amoebae}, the dimension of $\Trop_X(Y)$ may be less than the dimension of $Y$. We give examples of this behavior in Section \ref{section_structure}.

Moreover, we study the geometry of the closure $\overline{Y}$ of $Y$ in $X$ expanding on results of Tevelev \cite{Tevelev_tropcomp}, who has treated the case of $X$ being a toric variety. Let $X'$ be a log-regular compactification $X'$ of $X$. Then the cone complex $\Sigma_X$ is naturally a subcomplex of $\Sigma_{X'}$ and we have $\Trop_X(Y)=\Trop_{X'}(Y)\cap \Sigma_X$ for every closed subset $Y\subseteq X_0$. 

\begin{theorem}\label{thm_tropcomplogreg}
Let $Y$ be a closed subset of $X_0$. 
\begin{enumerate}[(i)]
\item Then the closure $\overline{Y}$ of $Y$ in $X$ is proper over $k$ if and only if 
\begin{equation*}
\Trop_{X'}(Y)\subseteq \Sigma_X \ .
\end{equation*}
\item If $\Trop_{X'}(Y)\supseteq\Sigma_X$, then $\overline{Y}$ intersects all strata $E$ of $X$ and the intersection has the expected dimension, i.e. 
\begin{equation*}
\dim \overline{Y}\cap E=\dim Y - \codim_X E \ .
\end{equation*}
\end{enumerate}
\end{theorem}

Suppose now that $X=X(\Delta)$ is a toric variety defined by a rational polyhedral fan $\Delta$ with big torus $T$ and $X'=X'(\Delta')$ a toric compactification of $X$, i.e. given by $\Delta$ being a subfan of a complete fan $\Delta'$. We are going to see in Proposition \ref{prop_troptoric} that the tropicalization $\Trop_{X'}(Y)$ of a closed subset $Y\subseteq T$ coincides with the usual tropicalization $\Trop(Y)$ of $Y$ as defined e.g. in \cite{Gubler_guide}. Moreover we will see in Example \ref{example_conecomplex=fan} the cone complexes $\Sigma_X$ and $\Sigma_{X'}$ can be naturally identified with $\Delta$ and $\Delta'$. 

In this case Theorem \ref{thm_tropcomplogreg} (i) reads as follows: The closure $\overline{Y}$ of $Y$ in $X$ is proper over $k$ if and only if $\Trop(Y)\subseteq\Delta$. We note that this is exactly the fundamental Proposition 2.3 of \cite{Tevelev_tropcomp}. Part (i) and (ii) of Theorem \ref{thm_tropcomplogreg} now imply the following: If $\Trop(Y)=\Delta$, then $\overline{Y}$ intersects every orbit of $X$ and the intersection has the expected dimension. This is one direction of \cite[Theorem 2.4]{Hacking_homology}. 

Using Nagata's embedding theorem and logarithmic resolution of singularities we obtain the following application of our results.

\begin{theorem}\label{thm_compsubvar}
Suppose $k$ is of characteristic $0$. Let $Y$ be a closed subset of a smooth variety $X_0$.  Then there is a partial log-regular compactification $X$ of $X_0$ such that 
\begin{enumerate}[(i)]
\item the closure $\overline{Y}$ of $Y$ in $X$ is proper over $k$, and 
\item the intersection of $\overline{Y}$ with every stratum $E$ of $X$ is non-empty and has the expected dimension. 
\end{enumerate}
\end{theorem}

Suppose now that $k=\C$. In \cite{Hacking_homology} Hacking exhibits that, in good cases, the simplicial cohomology of the link of a tropical variety recovers the weight zero part of the cohomology in the sense of \cite{Deligne_hodgeII}. In our situation the following result holds.

\begin{theorem}\label{thm_logregW0}
Let $X$ be a log-regular variety over $\C$. Then there are natural isomorphisms
\begin{equation*}
\tilde{H}^{i-1}\big(\Link(\Sigma_X),\Q\big)\simeq W_0H_c^i(X_0,\mathbb{Q})
\end{equation*}
for all $i\geq 0$. 
\end{theorem}

Now consider a \emph{sch\"on subvariety} $Y$ of a split algebraic torus $T$ over $\C$, as defined in \cite[Definition 1.3]{Tevelev_tropcomp}. By \cite[Theorem 1.4]{Tevelev_tropcomp} this means that for every tropical compactification $\overline{Y}$ in a $T$-toric variety $X$ the multiplication map 
\begin{equation*}
\mu\mathrel{\mathop:}T\times\overline{Y}\longrightarrow X
\end{equation*}
is smooth. Since $X$ is log-regular and $\mu$ is smooth, \cite[Theorem 8.2]{Kato_toricsing} implies that $Y$, endowed with the logarithmic structure induced from $X$, is log-regular. Therefore we obtain the following Corollary \ref{cor_schoenW0}. 

\begin{corollary}\label{cor_schoenW0}
If $Y$ is a sch\"on variety, then there are natural isomorphisms
\begin{equation*}
\tilde{H}^{i-1}\big(\Link(\Sigma_{\overline{Y}}),\Q\big)\simeq W_0H_c^i(Y,\mathbb{Q})
\end{equation*}
for all $i\geq 0$. 
\end{corollary}

In the above situation by \cite[Theorem 1.1]{Ulirsch_functroplogsch} we have a piecewise $\Z$-linear morphism $\Sigma_{\overline{Y}}\rightarrow \Sigma_X$, whose image is $\Trop_X(Y)$. In general, this morphism is not injective (see \cite[Example 7.4]{Ulirsch_functroplogsch}). But in the special situation that the (necessarily smooth) intersections of $\overline{Y}$ with the $T$-orbits in $X-T$ are irreducible the cone complex $\Sigma_Z$ is isomorphic to $\Trop_X(Y)$. In this case Corollary \ref{cor_schoenW0} reduces to Hacking's Theorem \cite[Theorem 2.5]{Hacking_homology}, namely that there are natural isomorphisms 
\begin{equation*}
\tilde{H}^{i-1}\big(\Link\big(\Trop_X(Y)\big),\Q\big)\simeq W_0H_c^i(Y,\mathbb{Q}) \ .
\end{equation*}

Our approach to the proof of Theorem \ref{thm_logregW0} makes use of the theory of dual complexes of simple normal crossing compactifications of smooth varieties as developed in \cite{Payne_boundarycomplexes}. Suppose now that in the situation of Corollary \ref{cor_schoenW0} the tropical compactification $\overline{Y}$ is smooth. In this case $D=\overline{Y}-Y$ has simple normal crossings and the tropical variety $\Trop_{\overline{Y}}(Y)=\Sigma_Y$ contains the full dual complex $\Delta(Y,D)$ of $D$ as its link (see Lemma \ref{lemma_dualcomplex=link}), while in $\Trop_X(Y)$ the connected components in the intersection of $\overline{Y}$ with a $T$-orbit of $X$ are collapsed to one cone. This explains why Hacking's Theorem \cite[Theorem 2.5]{Hacking_homology} needs the extra assumption that all intersections of $\overline{Y}$ with the $T$-orbits in $X-T$ are irreducible. 

Further details of our approach to tropicalization are explained in \cite{Ulirsch_functroplogsch}, which also treats the more general situation of fine and saturated \'etale logarithmic schemes and, in particular, the case of self-intersection. It is closely related to Thuillier's work \cite{Thuillier_toroidal} on the non-Archimedean geometry of toroidal embeddings, as well as to Popescu-Pampu and Stepanov's work \cite{PopescuPampu_localtrop} on local tropicalization. For the former we also refer the reader to \cite[Section 2 and 5]{AbramovichCaporasoPayne_tropmoduli}.

An alternative approach to dual complexes and its relation to the weight zero parts of the weight filtration can be found in \cite{ArapuraBakhtaryWlodarczyk_weightsdualcomplex}. Weight zero parts of weight filtration and its geometric interpretation have also been studied in \cite{Berkovich_Tateanalogue} and \cite{Berkovich_weightzero} from the point of view of Berkovich analytic geometry. 

A closely related construction is the \emph{parametrizing complex} in \cite[Section 4]{HelmKatz_monodromy} and in \cite[Section 2]{KatzStapledon_motivicnearbyfiber} of a tropical degeneration of a sch\"on variety $Y$ over a non-trivially valued field. It arises as the dual complex of the special fiber and its cohomology is naturally isomorphic to the weight zero part of the \'etale cohomology of $Y$. The main difference from our approach is that the authors of \cite{HelmKatz_monodromy} and \cite{KatzStapledon_motivicnearbyfiber} work over a non-trivially valued field and, in \cite{HelmKatz_monodromy}, use the weight filtration on the \'etale cohomology of $Y$. Although the situation considered in \cite{HelmKatz_monodromy} and \cite{KatzStapledon_motivicnearbyfiber} and the theory developed in this article are seemingly different, since we are only working in the constant coefficient case, the author expects that a theory of tropicalization for log-smooth varieties over valuation rings of rank one will be the right framework to generalize results such as \cite[Theorem 6.1]{HelmKatz_monodromy} to the logarithmic setting.

\subsection{Conventions}

We work over an algebraically closed field $k$ that is endowed with the trivial norm. A variety $X$ over $k$ is a separated integral scheme of finite type over $k$. A logarithmic structure on $X$ will always be defined in the Zariski topology. 

We can associate to $X$ its analytification $X^{an}$ in the sense of Berkovich \cite{Berkovich_book}, but also the analytic domain $X^\beth$ in $X^{an}$ defined by Thuillier in \cite{Thuillier_toroidal}. If $X=\Spec A$ is affine $X^{an}$ is the set of seminorms $\vert.\vert_x$ on $A$ extending the trivial norm on $k$, and $X^\beth$ is the subspace of bounded seminorms, i.e. seminorms $\vert.\vert_x$ that fulfill $\vert a\vert_x\leq 1$ for all $a\in A$. There is a continuous structure morphism $\rho\mathrel{\mathop:}X^{an}\rightarrow X$ and an anti-continuous reduction map $r\mathrel{\mathop:}X^\beth\rightarrow X$. If $X=\Spec A$ is affine, the map $\rho$ sends $x\in X^{an}$ to the prime ideal $\ker\vert.\vert_x$ in $A$ and $r$ sends $x\in X^\beth$ to the prime ideal $\big\{a\in A\big\vert \vert a\vert < 1\big\}$.

We write $\R_{\geq 0}$ for the additive monoid of non-negative real numbers and $\Rbar_{\geq 0}$ for its extension $\R_{\geq 0}\cup\{\infty\}$ with the obvious addition rules. A monoid $P$ is said to be \emph{fine}, if it is finitely generated and the canonical homomorphism $P\rightarrow P^{gp}$ into the associated group is injective. In this case $P$ is said to be \emph{saturated} if $n\cdot p\in P$ for some $p\in P^{gp}$ and $n\in\N_{>0}$ already implies $p\in P$. A fine and saturated monoid $P$ is called \emph{toric}, if it contains no non-trivial torsion elements, and \emph{sharp}, if it contains no non-trivial units. The subgroup of \emph{units} of a monoid $P$ will be denoted by $P^\ast$. 

A \emph{locally monoidal space} is a topological space $X$ together with a sheaf of monoids $\calO_X$. A \emph{morphism} of monoidal spaces consists of a continuous map $f\mathrel{\mathop:}X\rightarrow Y$ and a morphism of sheaves of monoids $f^\dagger\mathrel{\mathop:}f^\ast\calO_Y\rightarrow \calO_X$ such that $(f^\dagger)_x(\mathfrak{m}_{Y,f(x)})\subseteq\mathfrak{m}_{X,x}$ for all $x\in X$, where $\mathfrak{m}_{X,x}=\calO_{X,x}-\calO_{X,x}^\ast$ and $\mathfrak{m}_{Y,y}=\calO_{Y,y}-\calO_{Y,y}^\ast$ denote the unique maximal ideals of $\calO_{X,x}$ and $\calO_{Y,y}$ respectively. One can think of the notion of a locally monoidal space as a generalization of the notion of a locally ringed space. In this article every locally ringed space $(X,\mathcal{O}_X)$ will be implicitly thought of as a monoidal space with respect to multiplication on the structure sheaf.
A locally monoidal space $(X,\calO_X)$ is said to be \emph{sharp}, if the monoid $\calO_{X,x}$ is sharp for $x\in X$. Note that we can functorially associate to every monoidal space $(X,\calO_X)$ a sharp monoidal space $(X,\overline{\calO}_X)$ by setting $\overline{\calO}_X=\calO_X/\calO_X^\ast$.

\subsection{Acknowledgements}

The author would like to express his gratitude to Dan Abramovich for his constant support and encouragement. He would also like to thank Amaury Thuillier for suggesting the use of Kato fans and many insightful remarks on an earlier version of this paper. Thanks are also due to Walter Gubler for suggesting the technique used in the proof of Lemma \ref{lemma_tropopendense}, to Jenia Tevelev for sharing his unpublished notes on the subject of \cite{Tevelev_tropcomp}, and to the anonymous referee for several helpful remarks that significantly improved the readability of this article. The author also profited from conversations with Kenny Ascher, Lorenzo Fantini, Jeffrey Giansiracusa, Noah Giansiracusa, Angela Gibney, Diane MacLagan, Samouil Molcho, Johannes Nicaise, Sam Payne, Michael Temkin, and Jonathan Wise. 

Parts of this research have been carried out while enjoying the hospitality of Hebrew University, Jerusalem, and the University of Regensburg.


\section{Log-regular varieties}

\subsection{Log-regular varieties} Fix an algebraic variety $X$ together with a fine and saturated logarithmic structure $(M_X,\alpha_X)$ in the sense of \cite{Kato_logstr} consisting of a sheaf $M_X$ of monoids on $X$ and a morphism $\alpha\mathrel{\mathop:}M_X\rightarrow\mathcal{O}_X$ inducing an isomorphism $\alpha^{-1}(\mathcal{O}_X^\ast)\simeq\mathcal{O}_X^\ast$. For $x\in X$ denote by $I(M_X,x)$ the ideal in $\mathcal{O}_{X,x}$ generated by the image of $M_{X,x}-M_{X,x}^\ast$. 

Recall from \cite{Kato_toricsing} that $X$ is said to be \emph{logarithmically regular} (or short: \emph{log-regular}), if for all $x\in X$:
\begin{enumerate}[(i)]
\item the quotient $\mathcal{O}_{X,x}/I(M_X,x)$ is a regular local ring, and
\item we have $\dim\big(\mathcal{O}_{X,x}\big)=\dim\big(\mathcal{O}_{X,x}/I(M_X,x)\big) +\rk\big(M_x^{gp}/\mathcal{O}_{X,x}^\ast\big)$. 
\end{enumerate}
Since $k$ is perfect, by \cite[Proposition 8.3]{Kato_toricsing} this is equivalent to $X$ being logarithmically smooth over $\Spec k$ endowed with the trivial logarithmic structure $k^\ast$. 

\begin{example}
Let $N$ be a free and finitely generated abelian group. Consider a toric variety $X=X(\Delta)$ defined by a rational polyhedral fan $\Delta$ in $N_\mathbb{R}=N\otimes_\mathbb{Z}\mathbb{R}$ with big torus $T=N\otimes_\mathbb{Z}\mathbb{G}_m$, as explained in \cite{Fulton_toricvarieties}. The homomorphism $P\rightarrow k[P]$ on a $T$-invariant open affine subset $\Spec k[P]$ of $X$, for a toric moniod $P$ defined by a cone in $\Delta$, are compatible with restriction and induce a logarithmic structure $M_\Delta$ on $X$ making $X$ log-regular by \cite[Proposition 3.4]{Kato_logstr}. 
\end{example}

\begin{example}\label{example_divlogstr}
Let $X$ be a smooth variety and $D$ a divisor on $X$ with simple normal crossings. Then the \emph{divisorial logarithmic structure} $M_D$ associated to $D$, defined by
\begin{equation*}
M_D(U)=\big\{f\in\mathcal{O}(U)\big\vert f\vert_{U-D}\in\mathcal{O}_X^\ast(U-D)\big\}
\end{equation*}
on an open subset $U$ of $X$ makes $X$ into a log-regular variety by \cite[Proposition 4.7]{Thuillier_toroidal}. 
\end{example}

\subsection{Kato fans and characteristic morphism} Suppose that $X$ is log-regular. Consider the subspace \begin{equation*}
F_X=\big\{\xi\in X\big\vert I(M_X,\xi)=\mathfrak{m}_{X,\xi}\big\}
\end{equation*}
of $X$, where $\mathfrak{m}_{X,\xi}$ denotes the unique maximal ideal in the local ring $\mathcal{O}_{X,\xi}$. Endow $F_X$ with the sheaf $\mathcal{O}_{F_X}$ that is the restriction of $M_X/\mathcal{O}_{X}^\ast\simeq M_X/M_X^\ast$ to $F_X\subseteq X$. There is  a morphism 
\begin{equation*}
\phi_X\mathrel{\mathop:}(X,\overline{\mathcal{O}}_X)\longrightarrow F_X \ ,
\end{equation*} called the \emph{characteristic morphism} of $X$, that is given by sending $x\in X$ to the point $\phi(x)\in F_X$ corresponding to the prime ideal $I(M_X,x)$ in $\mathcal{O}_{X,x}$. Here we write $\overline{\mathcal{O}}_X$ for the sheaf $\mathcal{O}_X/\mathcal{O}_X^\ast$.  

By \cite[Proposition 10.1]{Kato_toricsing} $F_X$ has the structure of a fine and saturated \emph{Kato fan}, i.e. $F_X$ is a quasicompact sharp monoidal space in which every point $\xi\in F_X$ is contained in an \emph{affine open subset}, i.e. in an open subset $U\subseteq F_X$ that is isomorphic to the \emph{spectrum} $\Spec P$ of a fine and saturated monoid $P$. The spectrum $\Spec P$ is the set of prime ideals $\mathfrak{p}$ in $P$, endowed with the topology generated by the sets 
\begin{equation*}
D(f)=\big\{\mathfrak{p}\in\Spec P\big\vert f\notin\mathfrak{p}\big\}
\end{equation*}
 for $f\in P$ and a structure sheaf determined by 
\begin{equation*}
D(f)\mapsto P_f/P_f^\ast \ ,
\end{equation*}
where $P_f$ denotes the localization of $P$ with respect to the submonoid $\{n\cdot f\vert n\in\mathbb{N}\}$. We refer to \cite[Section 9]{Kato_toricsing} as well as to \cite[Section 3.5]{GabberRomero_foundationsalmostring} and \cite[Section 3]{Ulirsch_functroplogsch} for further details on the theory of Kato fans. 

\begin{example}
If $X$ is a toric variety with big torus $T$, then $F_X$ is the set of generic points of $T$-orbits in $X$. The characteristic morphism sends a point $x$ in $X$ to the generic point of the unique $T$-orbit containing $x$. Moreover, if $X=\Spec k[P]$ is an affine toric variety, defined by a toric monoid $P$, then $F_X=\Spec P$ and $\phi_X$ is induced by the homomorphism $P\rightarrow k[P]$. 
\end{example}

\subsection{Toroidal stratification}
The points in $F_X$ define a stratification of $X$ by locally closed subsets. Its strata are given by
\begin{equation*}
E(\xi)\mathrel{\mathop:}=\overline{\{\xi\}}-\bigcup_{\xi\rightarrow \xi'}\overline{\{\xi'\}}
\end{equation*}
for $\xi\in F_X$, where $\xi'$ runs through all strict specializations of $\xi$ in $F_X$ (written as $\xi\rightarrow \xi'$) excluding $\xi$ itself and $\overline{\{x\}}$ denotes the closure of a point $x\in X$. Write $X_0$ for the unique open stratum. It is generated by the generic point of $X$.

Given $\xi\in F_X$, a $\xi$-\emph{small} open subset $U$ in $X$ is an open neighborhood $U$ of $\xi$ that is contained in 
\begin{equation*}
X(\xi)=\bigcup_{\xi''\rightarrow\xi}E(\xi'') \ ,
\end{equation*}
where $\xi''$ runs through all $\xi''\in F_X$ that specialize to $\xi$, including $\xi$ itself. 

\begin{example}
If $X=X(\Delta)$ is a toric variety with big torus $T$, then the strata of $X$ are precisely the $T$-orbits. A covering by small open neighborhoods is given by the $T$-invariant open affine subsets of $X$ corresponding to cones $\sigma\in\Delta$.  
\end{example}
 
\subsection{Toric charts}
We are now going to apply Kato's criterion for log-smoothness \cite[Theorem 3.5]{Kato_logstr} (also see \cite[Chapter 4]{Kato_logsmoothdeformation}): Since $k$ is perfect, for every point $x\in X$ there is an open neighborhood $U$, a toric monoid $P$, and a chart $P\rightarrow\mathcal{O}_U$ of the logarithmic structure inducing an \'etale morphism $\gamma\mathrel{\mathop:}U\rightarrow Z$ into the toric variety $Z=\Spec k[P]$ over the big torus $T=\Spec k[P^{gp}]$. Note that $X$ is normal by \cite[Theorem 4.1]{Kato_toricsing}. Since $\gamma^{-1}(T)=X_0\cap U$, the open and dense embedding $X_0\hookrightarrow X$ is a simple toroidal embedding in the sense of \cite[Section 3]{Thuillier_toroidal}. 

\begin{definition}
Given $\xi\in F_X$, a \emph{toric chart} $\gamma\mathrel{\mathop:}U\rightarrow Z$ as above is said to be \emph{$\xi$-small}, if $U\subseteq X(\xi)$ and $Z=Z\big(\gamma(\xi)\big)$.
\end{definition}

\begin{lemma}\label{lemma_charmorchart}
Let $\gamma\mathrel{\mathop:}U\rightarrow Z=\Spec k[P]$ be a $\xi$-small chart. Then $\gamma$ induces an isomorphism  $\gamma^F\mathrel{\mathop:}F_U=U\cap F_X\simeq F_Z\xrightarrow{\sim}\Spec P$ such that the diagram
\begin{equation*}\begin{CD}
(U,\overline{\mathcal{O}}_U)@>\phi_U >>F_U\\
@V\gamma VV @V\gamma^F V\simeq V\\ 
(Z,\overline{\mathcal{O}}_Z)@>\phi_Z>> F_Z
\end{CD}\end{equation*}
is commutative. 
\end{lemma}

\begin{proof}
Since $U\subseteq X(\xi)$ the fan $F_U$ is naturally homeomorphic to $\Spec P$. Noting that the chart $P\rightarrow M_U$ induces an isomorphism $P/P^\ast\simeq\overline{M}_{U,\xi}$ now yields $F_U\simeq\Spec P=F_Z$. The commutativity of the above diagram follows from the natural factorization $P\rightarrow k[P]\rightarrow \mathcal{O}_U$ of the chart $P\rightarrow\mathcal{O}_U$ inducing $\gamma$.
\end{proof}

\subsection{Toroidal modifications} Given a fine and saturated Kato-fan $F$, recall from \cite[Section 9]{Kato_toricsing} that a \emph{proper subdivision} of $F$ is a fine and saturated Kato fan $F'$ together with a morphism $f\mathrel{\mathop:}F'\rightarrow F$ such that
\begin{enumerate}[(i)]
\item for all $x'\in F'$ the induced morphism $\mathcal{M}_{F,f(x')}^{gp}\rightarrow\mathcal{M}_{F',x'}^{gp}$ is surjective, 
\item the set $f^{-1}(x)$ is finite for all $x\in F$, and
\item the induced map $\Hom(\Spec\N,F')\rightarrow\Hom(\Spec\N,F)$ is bijective. 
\end{enumerate}

Let $X$ be a log-regular variety and denote by $F=F_X$ its associated Kato fan. Given a proper subdivision $F'\rightarrow F$, by \cite[Proposition 9.9, Proposition 9.11, and Proposition 10.3]{Kato_toricsing} there is a log-regular variety $X'$ together with a proper birational logarithmic morphism $X'\rightarrow X$ such that the diagram
\begin{equation*}\begin{CD}
(X',\overline{\mathcal{O}}_{X'})@>\phi_{X'}>> F'\\
@VVV @VVV\\
(X,\overline{\mathcal{O}}_X)@>\phi_X>> F
\end{CD}\end{equation*}
is commutative. The modification $X'\rightarrow X$ is final among those logarithmic morphisms giving rise to a diagram as above and induces an isomorphism $X_0'\simeq X_0$. We are going to refer to the logarithmic morphism $X'\rightarrow X$ as the \emph{toroidal modification} of $X$ induced by the proper subdivision $F'\rightarrow F$. 

\subsection{A glimpse on Artin fans} It is an interesting question whether the well-known notions of \emph{tropical pairs} \cite[Definition 1.1]{Tevelev_tropcomp}, \emph{sch\"on varieties} \cite[Definition 1.3]{Tevelev_tropcomp}, and \emph{initial degenerations} \cite[Section 5]{Gubler_guide} admit generalizations to our framework. These notions arise naturally in the theory of \emph{Artin fans}, a notion that can be implicitly traced back to \cite[Section 5]{Olsson_loggeometryalgstacks} but has first been introduced in \cite[Section 2.1 and 2.2]{AbramovichWise_invlogGromovWitten} and \cite[Section 3]{AbramovichChenMarcusWise_boundedness}. The goal of this section is not to introduce this theory in full generality, but to indicate its potential usefulness in the study of tropical compactifications. The general theory will be treated elsewhere.

Suppose that $X$ is a log-regular variety. The essential idea of \cite{AbramovichWise_invlogGromovWitten} and \cite{AbramovichChenMarcusWise_boundedness} is to lift the characteristic morphism $\phi_X\mathrel{\mathop:}(X,\mathcal{O}_X/\mathcal{O}_X^\ast)\rightarrow F_X$ to a morphism $\widetilde{\phi}_X\mathrel{\mathop:}X\rightarrow\mathcal{A}_X$ that is an initial factorization of the natural morphism $X\rightarrow\LOG_k$ into the stack $\LOG_k$ of logarithmic structures over $\Spec k$ as constructed in \cite{Olsson_loggeometryalgstacks}. If $X$ is a toric variety with big torus $T$, then its Artin fan $\calA_X$ is given by the stack quotient $[X/T]$.

A pair consisting of a log-regular variety $X$ and a closed subset $Y$ of $X_0$ is said to be a \emph{tropical pair} if the closure $\overline{Y}$ of $Y$ in $X$ is proper over $k$ and the natural morphism $\overline{Y}\rightarrow\calA_X$ is faithfully flat. In this case the \emph{initial degenerations} of $Y$ could be defined as the images  in $\calA_X$ of the intersections of $\overline{Y}$ with the toroidal strata. Moreover, the pair $(Y,X)$ is said to be \emph{sch\"on} if the morphism $\overline{Y}\rightarrow\calA_X$ is smooth. This notion generalizes the notion of sch\"on subvarieties of algebraic tori by the following Proposition \ref{prop_schoen=schoen}.

\begin{proposition}\label{prop_schoen=schoen}
Let $Y$ be a subvariety of an algebraic torus and $\overline{Y}$ be a tropical compactifcation in a toric variety $X$. Then the following properties are equivalent:
\begin{enumerate}[(i)]
\item The morphism $T\times \overline{Y}\rightarrow X$ induced by the $T$-action on $X$ is smooth, i.e. the very affine variety $Y$ is sch\"on in the sense of \cite[Definition 1.3]{Tevelev_tropcomp}.
\item The pair $(X,Y)$ is sch\"on in the sense defined above, i.e. the morphism $\overline{Y}\rightarrow \calA_X$ is smooth.
\item With the logarithmic structure induced from $X$ the closure $\overline{Y}$ of $Y$ in $X$ is log-regular. 
\end{enumerate}
\end{proposition}

\begin{proof}
We note that $\overline{Y}$ being log-regular is equivalent to being log-smooth over $\Spec k$ by \cite[Proposition 8.3]{Kato_toricsing}. This is, in turn, equivalent to the natural morphism $Y\rightarrow\LOG_k$ constructed in \cite{Olsson_loggeometryalgstacks} being smooth, which is equivalent to $Y\rightarrow\calA_X$ being smooth, since $\calA_X\rightarrow\LOG_k$ is \'etale by \cite[Proposition 2.1.1]{AbramovichWise_invlogGromovWitten} and the logarithmic structure on $\overline{Y}$ is the one induced by $X$. This proves the equivalence of part (ii) and part (iii).

We have already seen in the introduction that the smoothness of the morphism $T\times \overline{Y}\rightarrow X$ implies that $\overline{Y}$ is log-regular, i.e. that part (i) implies part (iii). Conversely, if $\overline{Y}\rightarrow\calA_X$ is smooth, then also the natural morphism $(\overline{Y}\rightrightarrows \overline{Y})\rightarrow (T\times X\rightrightarrows X)$ from the trivial groupoid representation of $\overline{Y}$ into the natural groupoid presentation of $\calA_X=[X/T]$ is smooth. The target arrow in $(T\times X\rightrightarrows X)$ is, however, nothing but the operation of $T$ on $X$ and thus the morphism $T\times \overline{Y}\rightarrow X$ is smooth. Therefore part (ii) implies part (i) and this finishes the proof.
\end{proof}


\section{Tropicalization}

\subsection{Cone complexes} Fine and saturated Kato fans were used by K. Kato \cite{Kato_toricsing} as a replacement of the notion of a \emph{rational polyhedral cone complex} (or short: \emph{cone complex}) as defined in \cite{KKMSD_toroidal}. For a fine and saturated Kato fan $F$ the underlying set of the associated \emph{cone complex} $\Sigma_F$ is equal to the set of morphisms $\Spec\mathbb{R}_{\geq 0}\rightarrow F$. 
 
In order to recover its structure as a cone complex consider the \emph{reduction map} $r\mathrel{\mathop:}\Sigma_F\rightarrow F$ that is given by sending a morphism $u\mathrel{\mathop:}\Spec \mathbb{R}_{\geq 0}\rightarrow F$ to the point $u(\mathbb{R}_{>0})\in F$. 

\begin{proposition}\label{prop_conecomplexreduction}
Suppose that $F$ is a fine and saturated Kato fan. Then the inverse image $r^{-1}(U)$ of an open affine subset $U=\Spec P$ in $\Sigma_F$ is a strictly convex rational polyhedral cone $\sigma_U=\Hom(P,\mathbb{R}_{\geq 0})$ and its relative interior $\mathring{\sigma}_U$ is given by $r^{-1}(x)$ for the unique closed point $x$ in $U$. 
\begin{enumerate}[(i)]
\item If $V\subseteq U$ for open affine subsets $U,V\subseteq F$, then $\sigma_V$ is a face of $\sigma_U$. 
\item The intersection $\sigma_U\cap\sigma_V$ is a union of finitely many common faces. 
\end{enumerate}
\end{proposition}

\begin{proof} 
Let $U=\Spec P$ be  open affine subset of $F$. Then 
\begin{equation*}
r^{-1}(U)=\Hom(\Spec \mathbb{R}_{\geq 0},U)=\Hom(P,\mathbb{R}_{\geq 0})=\Hom(P/P^\ast,\mathbb{R}_{\geq 0})
\end{equation*}
and this is a strictly  convex rational polyhedral cone $\sigma_U$. Its relative interior $\mathring{\sigma}_U$ is equal to 
\begin{equation*}
\big\{u\in\Hom(P,\mathbb{R}_{\geq 0})\big\vert u(p)>0 \ \ \forall p\in P-P^\ast\big\}=r^{-1}(x) 
\end{equation*}
for the unique closed point $x$ of $U=\Spec P$. 

For (i) let $y\in U=\Spec P$ be the unique closed point of $V$ and denote by $\mathfrak{p}_y$ the corresponding prime ideal. Then $\mathcal{M}_{F,y}=P_{\mathfrak{p}_y}/P_{\mathfrak{p}_y}^\ast$ and $r^{-1}(V)=\Hom(P_{\mathfrak{p}_y}/P_{\mathfrak{p}_y}^\ast,\mathbb{R}_{\geq 0})=\Hom(P_{\mathfrak{p}_y},\mathbb{R}_{\geq 0})$ is a face of $\Hom(P,\mathbb{R}_{\geq 0})$.

For (ii) it is enough to note that the intersection of two open affine subsets $U$ and $V$ is the union of finitely many open affines, since $F$ only contains finitely many open affine subsets and those generate the topology of $F$.
\end{proof}

The cone complex $\Sigma_F$ is canonically endowed with the weak topology: A subset $A$ in $\Sigma_F$ is closed if and only if  the intersections $A\cap\sigma_U$ are closed for all cones $\sigma_U$ in $\Sigma_F$, or alternatively for all open affine subsets $U$ of $F$. 

\begin{definition}
The \emph{extended cone complex} $\overline{\Sigma}_F$ associated to $F$ is defined to be the set of morphisms $\Spec \Rbar_{\geq 0}\rightarrow F$.
\end{definition}

The extended cone complex $\overline{\Sigma}_F$ is a canonical compactification of the open and dense subspace $\Sigma_F$ of $\overline{\Sigma}_F$. The reduction map $r$ above extends to a map $r\mathrel{\mathop:}\overline{\Sigma}_F\rightarrow F$ and the statements of Proposition \ref{prop_conecomplexreduction} go through without any problems. The weak topology on $\overline{\Sigma}_F$ is defined analogously using the topology of pointwise convergence on $\Hom(P,\overline{\mathbb{R}}_{\geq 0})$ as a local model. 

In our framework it is desirable to think about the extended cone complex $\overline{\Sigma}_F$ of $F$ as an analogue of the Berkovich analytic space $X^\beth$ associated to a variety $X$. In particular there is also a continuous structure map $\rho\mathrel{\mathop:}\overline{\Sigma}_F\rightarrow F$ given by sending a morphism $u\mathrel{\mathop:}\Spec\Rbar_{\geq 0}\rightarrow F$ to the point $u\big(\{\infty\}\big)\in F$. We refer the reader to \cite[Section 3]{Ulirsch_functroplogsch} for further details on this construction. 

\begin{example}\label{example_conecomplex=fan}
If $F=F_X$ is the Kato-fan associated to a toric variety $X=X(\Delta)$ defined by a fan $\Delta\subseteq N_\mathbb{R}$, then there is a natural identification $\Sigma_X=\Delta$. In order to see this we can assume that $X=\Spec k[P]$ is an affine toric variety, defined by a cone $\sigma$, and in this case we have 
\begin{equation*}
\Sigma_X=\Hom(P,\mathbb{R}_{\geq 0})=\sigma \ .
\end{equation*}
\end{example}

\subsection{Tropicalization} Let $X$ be a log-regular variety and denote by $\phi_X\mathrel{\mathop:}(X,\overline{\mathcal{O}}_X)\rightarrow F_X$ its characteristic morphism. Write $\Sigma_X$ and $\overline{\Sigma}_X$ for the cone complex and the extended cone complex associated to $F_X$ respectively. 

In \cite[Section 6.1]{Ulirsch_functroplogsch} we define the \emph{tropicalization map} $\trop_X\rightarrow\overline{\Sigma}_X$ as follows: A point $x\in X^\beth$ can be represented by a morphism $\underline{x}\mathrel{\mathop:}\Spec R\rightarrow (X,\overline{\mathcal{O}}_X)$ for a valuation ring $R$ extending $k$, since $\vert a\vert =1$ for all $a\in R^\ast$. Then its image $\trop_X(x)$ in $\overline{\Sigma}_X=\Hom(\Spec\overline{\mathbb{R}}_{\geq 0},F_X)$ is given by the composition 
\begin{equation*}\begin{CD}
\Spec \overline{\mathbb{R}}_{\geq 0}@>\val^\#>>\Spec R @>\underline{x}>> (X,\overline{\mathcal{O}}_X)@>\phi_X>>F_X \ , 
\end{CD}\end{equation*}
where $\val^\#$ is the morphism $\Spec \overline{\mathbb{R}}_{\geq 0}\rightarrow\Spec R$ induced by the valuation $\val\mathrel{\mathop:}R\rightarrow\overline{\mathbb{R}}_{\geq 0}$ on $R$. 

Recall the following Proposition \ref{prop_tropwelldef} from \cite{Ulirsch_functroplogsch}.

\begin{proposition}[\cite{Ulirsch_functroplogsch} Proposition 6.2 (i)]\label{prop_tropwelldef}
The tropicalization map $\trop_X\mathrel{\mathop:}X^\beth\rightarrow\overline{\Sigma}_X$ is well-defined and continuous. Moreover, the diagrams
\begin{equation*}\begin{CD}
X^\beth @>\trop_X>>\overline{\Sigma}_X\\
@VrVV @VrVV\\
(X,\overline{\mathcal{O}}_X) @>\phi_X>> F_X
\end{CD}
\qquad
\begin{CD}
X^\beth @>\trop_X>>\overline{\Sigma}_X\\
@V\rho VV @V\rho VV\\
(X,\overline{\mathcal{O}}_X)@>\phi_X >>F_X
\end{CD}\end{equation*}
commute. 
\end{proposition}

\begin{corollary}[Strata-cone correspondence]\label{cor_stratacone}
There is an order-reversing one-to-one correspondence between the cones in $\Sigma_X$ and the strata of $X$. Explicitly it is given by 
\begin{equation*}
\mathring{\sigma}\longmapsto r\big(\trop_X^{-1}(\mathring{\sigma})\big)
\end{equation*}
for a an open cone $\mathring{\sigma}\subseteq\Sigma_X$ and
\begin{equation*}
E\longmapsto \trop_X\big(r^{-1}(E)\cap X_0^{an}\big)
\end{equation*}
for a stratum $E$ of $X$. 
\end{corollary}

\begin{proof}
This is an immediate consequence of the commutativity of 
\begin{equation*}\begin{CD}
X^\beth @>\trop_X>>\overline{\Sigma}_X\\
@Vr VV @Vr VV\\
X@>\phi_X >>F_X 
\end{CD}\end{equation*} 
from Proposition \ref{prop_tropwelldef}, Proposition \ref{prop_conecomplexreduction} and the fact that $\phi_X$ sends every point in a stratum $E=E(\xi)$ to its generic point $\xi$. 
\end{proof}

\begin{corollary}
The tropicalization map induces a continuous map $X^\beth\cap X_0^{an}\rightarrow\Sigma_F$. 
\end{corollary}

\begin{proof}
This follows from the commutativity of 
\begin{equation*}\begin{CD}
X^\beth @>\trop_X>>\overline{\Sigma}_X\\
@V\rho VV @V\rho VV\\
(X,\overline{\mathcal{O}}_X)@>\phi_X >>F_X
\end{CD}\end{equation*}
and the observations that $\rho^{-1}(X_0)=X^\beth\cap X^{an}$ as well as $\rho^{-1}(X_0\cap F_X)=\Sigma_F$. 
\end{proof}

\begin{definition}
Let $Y$ be a closed subset of $X_0$. The projection 
\begin{equation*}
\Trop_X(Y)=\trop_X(Y^{an}\cap X^\beth)\subseteq \Sigma_X
\end{equation*}
 is called the \emph{tropical variety} associated to $Y$ relative to $X$. 
\end{definition}

\begin{remark}
In \cite{Thuillier_toroidal} Thuillier constructs a strong deformation retraction $\mathbf{p}\mathrel{\mathop:}X^\beth\rightarrow X^\beth$ onto the \emph{skeleton} $\mathfrak{S}(X)$ of $X$. By \cite[Theorem 1.2]{Ulirsch_functroplogsch} there is a homeomorphism $J_X\mathrel{\mathop:}\overline{\Sigma}_X\rightarrow \mathfrak{S}(X)$ making the diagram
\begin{center}\begin{tikzpicture}
  \matrix (m) [matrix of math nodes,row sep=1em,column sep=2em,minimum width=2em]
  {  
  & & \mathfrak{S}(X) \\ 
  X^\beth&  &\\
  & & \overline{\Sigma}_X\\
  };
  \path[-stealth]
    (m-2-1) edge node [above] {$\mathbf{p}$} (m-1-3)
    (m-2-1) edge node [below] {$\trop_X$} (m-3-3)
    (m-3-3) edge node [right] {$J_X$} (m-1-3);
\end{tikzpicture}\end{center}
commute. In particular, this implies $\Trop_X(X_0)=\Sigma_X$.
\end{remark}

\subsection{The case of toric varieties}
Let $N$ be a free and finitely generated abelian group, denote by $N^\vee=\Hom(N,\mathbb{Z})$ its dual, and consider the torus $T=N\otimes_\mathbb{Z}\mathbb{G}_m=\Spec k[N^\vee]$. There is a well-known continuous tropicalization map 
\begin{equation*}\begin{split}
\trop\mathrel{\mathop:}T^{an}&\longrightarrow N_\mathbb{R}=\Hom(N^\vee,\R) \\
x&\longmapsto \big(u\mapsto -\log\vert \chi^u\vert_x\big) \ ,
\end{split}\end{equation*}
where $\chi^u$ denotes the character of $u\in N^\vee$ in the coordinate ring of $T=\Spec k[N^\vee]$. One can define the \emph{tropical variety} $\Trop(Y)$ associated to a closed subset $Y\subseteq T$ to be the image $\trop(Y^{an})$ of $Y^{an}$ in $N_\R$ (see e.g. \cite[Section 4]{Gubler_guide}). 

\begin{proposition}\label{prop_troptoric}
Suppose we are given a toric variety $X=X(\Delta)$ defined by an integral pointed polyhedral fan $\Delta$ in $N_\mathbb{R}$ with big torus $T$. Then we have
\begin{equation*}
\Trop_Z(Y)=\Trop(Y)\cap\Delta \ .
\end{equation*}
\end{proposition}

\begin{proof}
It is enough to note that the diagram
\begin{equation*}\begin{CD}
T^{an}\cap X^\beth @>\trop_X>>\Delta \\
@V\subseteq VV @V\subseteq VV\\
T^{an}@>\trop>>N_\mathbb{R}
\end{CD}\end{equation*}
is commutative.
\end{proof}

We are going to use the following technical Lemma in the proof of Proposition \ref{prop_Tropchart}.

\begin{lemma}\label{lemma_tropopendense}
Let $U$ be an open and dense subset of a closed subset $Y$ of an algebraic torus $T$. Then we have
\begin{equation*}
\Trop(Y)=\trop(U^{an}) \ .
\end{equation*}
\end{lemma}

\begin{proof}
The inclusion $\supseteq$ is trivial. For the converse inclusion consider an element $w\in\Trop(Y)$.  Our goal is to find a point $q\in U^{an}$ such that $\trop(q)=w$. 

We can find a non-trivially valued complete algebraically closed field $K$ extending $k$ such that $\trop(t)=w$ for an element $t\in Y(K)$. Passing to $t^{-1}Y_K$ we can assume $w=0$ and $t=e$, the identity element in $T(K)$. Write $K^\circ$ for the valuation ring of $K$ and $\tilde{K}$ for the residue field of $K^\circ$. 

Denote the canonical $K^\circ$-model of $T_K$ over $K_0$ by $\mathbb{T}_{K^\circ}$ and write $\mathcal{Y}_K$ for the closure of $Y_K$ in $\mathbb{T}_{K^\circ}$, i.e. the unique closed subscheme of $\mathbb{T}_{K^\circ}$ that has $Y_K$ as its generic fiber and is flat over $K^\circ$. By the fundamental theorem of tropical algebraic geometry the special fiber $\mathcal{Y}_{K,s}$ of $\mathcal{Y}_K$ over $\tilde{K}$, is non-empty (see e.g. \cite[Section 5]{Gubler_guide}). By \cite[Proposition 4.14]{Gubler_guide}  for every $\tilde{K}$-rational point $\tilde{q}$ of the special fiber $\mathcal{Y}_{K,s}$ there is a $K^\circ$-rational point $q$ of $\mathcal{Y}_K$ mapping to $\tilde{q}$ via the reduction map, that is also contained in $U$. Since $\tilde{q}\in\mathbb{T}_s(\tilde{K})=(\tilde{K}^\ast)^n$, all of the components of $\tilde{q}$ are nonzero. Therefore all components of $q$ have norm $1$, which implies $\trop(q)=0$. 
\end{proof}

\subsection{Structure of tropical varieties}\label{section_structure}

We are now going to prove Theorem \ref{thm_BieriGroves} based on the following description of the local structure of $\Trop_X(Y)$. 

Let $Y$ be a closed subset of the open stratum $X_0$ of a log-regular variety $X$. Given $\xi\in F_X$, let $\gamma\mathrel{\mathop:}U\rightarrow Z$ be a $\xi$-small chart into an affine toric variety $Z=\Spec k[P]$ with big torus $T=\Spec k[P^{gp}]$. As an immediate consequence of Lemma \ref{lemma_charmorchart} we obtain the following Lemma.

\begin{lemma}\label{lemma_tropchart}
The $\xi$-small toric chart $\gamma$ induces an isomorphism $\gamma^\Sigma\mathrel{\mathop:}\overline{\Sigma}_U\xrightarrow{\sim}\overline{\Sigma}_Z$ such that the diagram
\begin{equation*}\begin{CD}
U^\beth @>\trop_U>> \overline{\Sigma}_U\\
@V\gamma^\beth VV @V\gamma^\Sigma V\simeq V \\
Z^\beth @>\trop_Z>> \overline{\Sigma}_Z
\end{CD}\end{equation*}
commutes.
\end{lemma}

Denote the closure of $\gamma(U\cap Y)$ in $T$ by $\overline{Y}^T$.

\begin{proposition}\label{prop_Tropchart}
In the situation of Lemma \ref{lemma_tropchart} we have \begin{equation*}
\gamma^\Sigma\big(\Trop_U(Y)\big)=\Trop(\overline{Y}^T)\cap \Sigma_Z \ .
\end{equation*}
\end{proposition}

\begin{proof}
By Chevalley's theorem \cite[1.8.4]{EGAIVi} the subset $\gamma(Y\cap U)$ is constructible. Therefore Lemma \ref{lemma_tropopendense} implies
\begin{equation*}
\trop\big(\gamma(Y\cap U)^{an}\big)=\Trop(\overline{Y}^T) \ .
\end{equation*}
Applying Proposition \ref{prop_troptoric} and Lemma \ref{lemma_tropchart} finishes the proof. 
\end{proof}

\begin{proof}[Proof of Theorem \ref{thm_BieriGroves}]
Choose a covering $U_i$ of $X$ by $\xi_i$-small toric charts $\gamma_i\mathrel{\mathop:}U_i\rightarrow Z_i=\Spec k[P_i]$ for some $\xi_i\in F_X$ and set $T_i=\Spec k[P_i^{gp}]$. By Proposition \ref{prop_Tropchart} we have 
\begin{equation*}
\Trop_X(Y)=\bigcup_{i}\big(\gamma_i^\Sigma\big)^{-1}\big(\Trop(\overline{Y}^{T_i})\cap\Sigma_{Z_i}\big) \ .
\end{equation*}
using the isomorphisms coming from Lemma \ref{lemma_tropchart}.

The classical Bieri-Groves Theorem (see \cite[Theorem A]{BieriGroves_polyhedral} and \cite[Theorem 2.2.3]{EinsiedlerKapranovLind_amoebae}) states that all $\Trop(\overline{Y}^{T_i})$ have the structure of a rational polyhedral fan of dimension $\dim Y$. The intersections $\Trop(\overline{Y}^{T_i})\cap\Sigma_{Z_i}$ are therefore rational polyhedral fans of dimension $\leq \dim Y$ and, since the $\gamma_i^\Sigma$ are isomorphisms of rational polyhedral cone complexes, this implies that the tropical variety $\Trop_X(Y)$ itself carries the structure of a rational polyhedral cone complex of dimension $\leq\dim Y$. 
\end{proof}

\begin{example}\label{example_dimdroptoric}
Let $X=X(\Delta)$ be a toric variety defined by a rational polyhedral fan $\Delta$ with big algebraic torus $T$. Moreover suppose that $Y$ is a closed subset of $T$ of dimension greater than $\dim\Delta$. Then  Propositon \ref{prop_troptoric} implies that 
\begin{equation*}
\dim \Trop_X(Y)\leq\dim \Delta<\dim Y \ .
\end{equation*}
\end{example}

\begin{example}\label{example_dimdropP1xP1}
Suppose that $X=\mathbb{P}^1\times\mathbb{P}^1$ and $X_0$ is the complement of the smooth divisor $D=\{\infty\}\times\mathbb{P}^1$. As seen in Example \ref{example_divlogstr} the divisorial logarithmic structure $M_D$ makes $(X,M_D)$ into a log-regular variety. Consider now the $1$-dimensional closed subset $Y=\{0\}\times\mathbb{P}^1$ of $X_0$. The cone complex $\Sigma_X$ consists of a copy of $\mathbb{R}_{\geq 0}$ and using Tevelev's Lemma \ref{lemma_Tevelevslemma} we see that the tropicalization $\Trop_X(Y)$ is given by the origin of $\mathbb{R}_{\geq 0}$.
\end{example}

The above Examples \ref{example_dimdroptoric} and \ref{example_dimdropP1xP1} illustrate that in Theorem \ref{thm_BieriGroves} the dimension of $\Trop_X(Y)$ can be less than the dimension of $Y$. 

\subsection{Invariance under toroidal modifications} Let $F$ be a fine and saturated Kato fan. Proper subdivisions $F'\rightarrow F$ induce piecewise $\Z$-linear homeomorphisms $\Sigma_{F'}\simeq\Sigma_F$ 
and are in one-to-one correspondence with finite integral polyhedral subdivisions of the cone complex $\Sigma_F$.

\begin{proposition}
Let $X$ be a log-regular variety and let $X'\rightarrow X$ be a toroidal modification. Given a closed subset $Y\subseteq X_0$, under the piecewise $\mathbb{Z}$-linear homeomorphic identification $\Sigma_{X'}\simeq\Sigma_X$ we have 
\begin{equation*}
\Trop_X(Y)=\Trop_{X'}(Y)\ .
\end{equation*} 
\end{proposition}

\begin{proof}
The commutative diagram
\begin{equation*}\begin{CD}
(X',\overline{\mathcal{O}}_{X'})@>\phi_{X'}>> F'\\
@VVV @VVV\\
(X,\overline{\mathcal{O}}_X)@>\phi_X>> F
\end{CD}\end{equation*}
on the level of characteristic morphisms naturally induces a commutative diagram 
\begin{equation*}\begin{CD}
(X')^\beth @>\trop_{X'}>> \overline{\Sigma}_{X'}\\
@VVV @VVV\\
X^\beth @>\trop_X>>\overline{\Sigma}_X 
\end{CD}\end{equation*}
on the level of tropicalization maps. Since $X'_0\simeq X_0$, we have $(X_0')^{an}\simeq X_0^{an}$ and therefore $\Trop_{X'}(Y)=\Trop_X(Y)$.
\end{proof}


\section{Tropical compactification}

Our approach to the proof of Theorem \ref{thm_tropcomplogreg} (i) formally resembles the one of \cite[Proposition 2.3]{Tevelev_tropcomp}. In particular, we also have a version of \cite[Lemma 2.2]{Tevelev_tropcomp}, which is known as Tevelev's Lemma. 

\begin{lemma}[Tevelev's Lemma]\label{lemma_Tevelevslemma}
Let $Y$ be closed subset of $X_0$. The closure $\overline{Y}$ of $Y$ in $X$ intersects a stratum $E=E(\xi)$ if and only if the tropical variety $\Trop_X(Y)$ intersects the relatively open cone $\mathring{\sigma}=r^{-1}(\xi)$ corresponding to it.
\end{lemma}

\begin{proof} Our argument is essentially a version of the one presented \cite[Lemma 11.6]{Gubler_guide}. Assume there is a point $w$ in the intersection of $\Trop_X(Y)$ and $\mathring{\sigma}$. Then there is an element $x\in Y^{an}\cap X^\beth$ such that $\trop_Y(x)=w$ and $x$ will lie in $r^{-1}(E)$ by Corollary \ref{cor_stratacone}. Now $r(x)$ will be in $\overline{Y}$, since $\overline{Y}^{an}\cap X^\beth=\overline{Y}^\beth$. Thus $E\cap\overline{Y}$ is non-empty. 

Conversely suppose $\overline{Y}\cap E$ is non-empty. We can find a point $x$ in the intersection of $r^{-1}(E)$ and $Y^{an}\cap X^\beth$ and by Corollary \ref{cor_stratacone} the intersection of $\Trop_X(Y)$ and the relative interior of $\sigma$ is non-empty. 
\end{proof}

\begin{proof}[Proof of Theorem \ref{thm_tropcomplogreg}]
\framebox{Part (i):} Suppose $\overline{Y}$ is complete. Then it is already closed in $X'$ and cannot intersect a stratum of $X'-X$. By Lemma \ref{lemma_Tevelevslemma} the tropical variety $\Trop_{X'}(Y)$ does not intersect the relative interior of any cone $\sigma$ in $\Sigma_{X'}$ that is not contained in $\Sigma_X$. Therefore we have $\Trop_{X'}(Y)\subseteq\Sigma_X$. 

Conversely assume $\overline{Y}$ is not complete. Then the closure $\overline{Y}'$ of $Y$ in $X'$ has to intersect a stratum of the boundary divisor $X'-X$. By Tevelev's lemma this implies $\Trop_{X'}(Y)$ intersecting the relative interior of a cone $\sigma$ in $\Sigma_{X'}$ that is not contained in $\Sigma_X$. Thus $\Trop_{X'}(Y)$ is not a subset of $\Sigma_X$. 
 
\framebox{Part (ii):} Suppose $\Trop_{X'}(Y)$ contains $\Sigma_X$. Note first that the intersection of $\overline{Y}$ with all toroidal strata is non-empty by Lemma \ref{lemma_Tevelevslemma}. Let $\sigma=\sigma(\xi)$ be the cone in $\Sigma_X$ associated to $\xi\in F_X$. We want to show that 
\begin{equation*}
\dim \overline{Y}\cap E=\dim Y -\codim_XE \ ,
\end{equation*}
where $E=E(\xi)$ denotes the stratum $E$ corresponding to $\xi$. The inequality $\geq$ always holds; so we are left to show the converse inequality. 

Let $\gamma\mathrel{\mathop:}U\rightarrow Z=\Spec k[P]$ be a $\xi$-small chart into a an affine toric variety over the torus $T=\Spec k[P^{gp}]$. Set $N=\Hom(P^{gp},\Z)$. Then under the identifications of Lemma \ref{lemma_tropchart} we have $\sigma=\Sigma_Z$ as well as $\Trop(\overline{Y}^{T})\supseteq\Sigma_Z$ by Lemma \ref{prop_Tropchart}. Now choose a fan $\Delta\subseteq N_\mathbb{R}$ containing $\sigma$ such that $\Trop(\overline{Y}^T)=\Delta$ and consider the toric variety $Z'=Z'(\Delta)$ defined by $\Delta$. 

By part (i) the closure $\overline{Y}^{Z'}$ of $Y$ in $Z'$ is complete and so \cite[Theorem 2.4]{Hacking_homology} implies that $\overline{Y}^{Z'}$ intersects the toric stratum $E'=E(\xi')$ corresponding to $\xi'=\gamma(\xi)\in F_Z$ in the expected dimensions. Since $\gamma$ is \'etale, this implies 
\begin{equation*}\begin{split}
\dim \overline{Y}\cap E\cap U &\leq\dim\overline{Y}^{Z'}\cap E'\\
&=\dim\overline{Y}^{Z'} -\codim E'\\
&=\dim \overline{Y} -\codim E \ ,
\end{split}\end{equation*}
and this shows the desired inequality, since we can cover $\overline{Y}\cap E$ by $\xi$-small toric charts.
\end{proof}

\begin{proof}[Proof of Theorem \ref{thm_compsubvar}]
By Nagata's embedding theorem and logarithmic resolution of singularities we can find a compactification $X^\ast$ of $X_0$ such that $X^\ast-X_0$ has simple normal crossings. So, endowed with the logarithmic structure induced by the divisor $X^\ast-X_0$, the variety $X^\ast$ is log-regular. 

By Theorem \ref{thm_BieriGroves} $\Trop_{X^\ast}(Y)$ is a polyhedral cone complex with integral structure. So we can choose a proper subdivision $F'\rightarrow F^\ast$ of $F^\ast=F_{X^\ast}$ corresponding to a toroidal modification $X'\rightarrow X^\ast$ such that $\Trop_{X'}(Y)=\Trop_{X^\ast}(Y)$ is subcomplex of $\Sigma_{X'}$. Now let $X$ be the log-regular variety consisting of all strata $E=E(\xi)$ of $X'$ whose associated cone $\sigma=\sigma(\xi)$ is a subset of $\Trop_{X'}(Y)$, or equivalently (by Lemma \ref{lemma_Tevelevslemma}) all strata of $X'$ intersecting the closure $\overline{Y}$ of $Y$ in $X'$. By Theorem \ref{thm_tropcomplogreg} the log-regular variety $X$ has the desired properties. 
\end{proof}


\section{Dual complexes and the cohomology of links}

Let $D$ be a divisor with simple normal crossings on a smooth variety $X$. The \emph{dual complex} $\Delta(X,D)$ is the $\Delta$-complex in the sense of \cite[Section 2.1]{Hatcher_algtop} whose vertices correspond to irreducible components $D_1,\ldots, D_l$ of $D$ and whose $k$-dimensional simplices correspond to irreducible components of $k$-fold intersections of the $D_i$. By the strata-cone correspondence Corollary \ref{cor_stratacone} we have the following Lemma \ref{lemma_dualcomplex=link}.

\begin{lemma}\label{lemma_dualcomplex=link}
The link $\Link(\Sigma_{(X,D)})$ of the cone complex $\Sigma_{(X,D)}$ is equal to the dual complex of $\Delta(D)$.
\end{lemma}

For the proof of Lemma \ref{lemma_dualcomplex=link} we also refer to \cite[Proposition 4.7]{Thuillier_toroidal}).

Suppose now that $k=\C$. The following Lemma is essentially contained in \cite{Deligne_hodgeII}. We refer to the proof of \cite[Theorem 4.4]{Payne_boundarycomplexes} for an explicit proof of this fact.

\begin{lemma}[Deligne's Lemma]\label{lemma_Deligneslemma}
There are natural isomorphisms
\begin{equation*}
\tilde{H}^{i-1}\big(\Delta(X,D),\Q\big)\simeq W_0 H_c^i(X_0,\mathbb{Q})
\end{equation*}
for all $i\geq 0$. 
\end{lemma}

\begin{proof}[Proof of Theorem \ref{thm_logregW0}]
Suppose that $X$ is log-regular. By \cite[Proposition 9.8]{Kato_toricsing}, which, in turn, relies on \cite[Theorem 11*]{KKMSD_toroidal}, we can find a proper subdivision $F'\rightarrow F$ inducing a toroidal modification $X'\rightarrow X$ such that $X'$ is smooth. Then the divisor $D'=X'-X_0$ has simple normal crossings and Lemma \ref{lemma_dualcomplex=link} as well as Deligne's Lemma \ref{lemma_Deligneslemma} imply 
\begin{equation*}
\tilde{H}^{i-1}\big(\Link(\Sigma_{X'}),\Q\big)=\tilde{H}^{i-1}\big(\Delta(D'),\Q\big)\simeq W_0 H_c^i(X_0,\mathbb{Q}) \ .
\end{equation*}
Since the proper subdivision $F'\rightarrow F$ induces a proper subdivision $\Sigma_{F'}$ of $\Sigma_F$, it follows that $\Link(\Sigma_{X'})$ and $\Link(\Sigma_X)$ are homeomorphic. 
\end{proof}


\bibliographystyle{alpha}
\bibliography{biblio}{}

\end{document}